\documentclass[11pt]{amsart}

\usepackage{amsmath, amsfonts, amssymb, amsthm, xcolor, mathtools, mathrsfs, graphicx, hyperref}

\newtheorem{theorem}{Theorem}[section]
\newtheorem{conjecture}[theorem]{Conjecture}
\newtheorem{corollary}[theorem]{Corollary}

\theoremstyle{definition}

\numberwithin{theorem}{section}
\numberwithin{equation}{section}

\begin{document}

\title[Fixed Perimeter Analogues of Results Related to Parity]{Fixed Perimeter Analogues of Several Partition Results Related to Parity}

\author{Gabriel Gray}
\address{University of Dayton}
\email{grayg1@udayton.edu}
\author{Emily Payne}
\address{Oregon State University}
\email{payneemi@oregonstate.edu}
\author{Holly Swisher}
\address{Oregon State University}
\email{swisherh@oregonstate.edu}
\author{Ren Watson}
\address{University of Texas at Austin}
\email{renwatson@utexas.edu}

\thanks{The authors were supported by NSF grant DMS-2101906}

\begin{abstract}
In 2016, Straub proved that Euler's classic partition identity holds true for partitions with largest hook (perimeter) $n$. This inspired further study of the relationship between classical partitions and fixed perimeter partitions. We extend the study of parity bias inequalities, first introduced by Kim, Kim, and Lovejoy in 2020, to the fixed perimeter setting and show using combinatorial methods that fixed perimeter analogues of many classical parity bias results can be proven and generalized. We also extend these methods to prove similar inequalities for the fixed perimeter analogues of PED and POD partitions. We additionally develop recursive formulas for the number of perimeter $n$ partitions with odd parts distinct and even parts unrestricted and with even parts distinct and odd parts unrestricted.
\end{abstract}

\maketitle

\section{Introduction and Statement of Results}

An \textit{integer partition} is a finite, nonincreasing sequence of positive integers called \textit{parts}.  If the parts of a partition $\pi$ sum to $n$, we call $\pi$ a \emph{partition of} $n$.  The well-known partition function $p(n)$ counts the number of partitions of $n$, and there are many interesting and well-studied variants of this function which can be described by $p(n\mid *)$, which counts the number of partitions of $n$ satisfying some condition $*$.

We will write a partition $\pi$ with parts $\pi_1\geq \pi_2 \geq \ldots \geq \pi_m$ either by $\pi=(\pi_1,\pi_2,\ldots,\pi_m)$, or $\pi_1 + \pi_2 + \cdots + \pi_m$.   One can visualize a partition $\pi$ with a \textit{Ferrers diagram}, in which each part $\pi_i$ is represented by a left-justified row of $\pi_i$ dots ordered from top to bottom.   

For any partition $\pi$, we let $\alpha_\pi$ denote the largest part and $\lambda_\pi$ the number of parts. We define the \emph{perimeter} of a partition $\pi$ to be $$\alpha_\pi+\lambda_\pi -1,$$ which is the number of dots along the top row and left column of the Ferrers diagram of $\pi$, i.e., the largest hook length of $\pi$.  For example, the following is the Ferrers diagram of the partition $5+5+3$ which has perimeter $7$. 

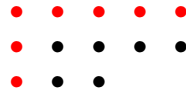
\begin{center}
    \begin{tabular}{c c c c c}
        \color{red} $\bullet$ & \color{red} $\bullet$ & \color{red} $\bullet$ & \color{red} $\bullet$ & \color{red} $\bullet$ \\
        \color{red} $\bullet$ & $\bullet$ & $\bullet$ & $\bullet$ & $\bullet$ \\
        \color{red} $\bullet$ & $\bullet$ & $\bullet$ &  & 
    \end{tabular} \\
\end{center} 

Analogous to $p(n)$, we define $r(n)$ to count the number of partitions of perimeter $n$.  We note that it is straightforward to confirm that $r(n)=2^{n-1}$ for $n \geq 0$.  We observe $p(4)=5$, since the partitions of $4$ are given by
\[
4, \enspace 3+1, \enspace 2+2, \enspace 2+1+1, \enspace 1+1+1+1, 
\]
and $r(4)=8$, since the partitions (of any $n$) that have perimeter $4$ are
\[
4, \enspace 3+1, \enspace 3+2, \enspace 3+3, \enspace 2+1+1, \enspace 2+2+1, \enspace 2+2+2, \enspace 1+1+1+1. 
\]

It is natural to ask whether identities that hold for partition functions of the form $p(n\mid *)$, which count partitions satisfying condition $*$ that have parts which sum to $n$, have fixed perimeter analogues in terms of the corresponding function $r(n\mid *)$, which count partitions satisfying condition $*$ that have perimeter $n$.  A priori there is no reason to believe that would be true.  However, in 2016 Straub \cite[Thm. 1.4]{straub2016core} showed that Euler's classic partition identity 
\begin{equation} \label{Euler}
p(n\mid \text{odd parts}) = p(n\mid \text{distinct parts})
\end{equation}
has a direct analogue for fixed perimeter partitions,
\begin{equation} \label{Euleranalogue}
    r(n\mid\text{odd parts})=r(n\mid\text{distinct parts}).
\end{equation}
This result has inspired further investigation into fixed perimeter analogues \cite{fu2018partitions, amdeberhan2022refinements, lin2022combinatorics, CHSS,   gray2026fixed}.

In this paper, we consider several recent partition results related to parity and investigate the analogous questions in the fixed perimeter setting.  While we focus here on parity, we note that bias for integer partitions with respect to congruences with other moduli have also been investigated \cite{BallantineFolsom, chern2022further}.

Throughout, we will use the notation $P(n\mid *)$ and $R(n\mid *)$ (and variations) to denote the sets of partitions counted by $p(n\mid *)$ and $r(n\mid *)$ (and variations), respectively.  For a set of positive integers $S$, let $p^S(n)$ denote the number of partitions of $n$ with no parts from $S$, and $r^S(n)$ denote the number of partitions with perimeter $n$ and no parts from $S$.   

We first consider the functions
\begin{align*}
    p_o(n)\coloneqq p(n \mid \text{more odd than even parts}),   \\
    p_e(n)\coloneqq p(n \mid \text{more even than odd parts}), 
\end{align*}
and the following refinements excluding parts from a set $S$, 
\begin{align*}    
    p_o^S(n)\coloneqq p(n \mid \text{more odd than even parts and no parts from } S),   \\
    p_e^S(n)\coloneqq p(n \mid \text{more even than odd parts and no parts from } S).     
\end{align*}

In 2020, Kim, Kim, and Lovejoy \cite[Thm. 1]{kim2020parity} proved that for all $n \neq 2$ (for which equality holds), 
\begin{align}\label{parityintro}
p_o(n) > p_e(n).
\end{align} 
This result naturally motivates the use of the term \textit{parity bias}. 

Let $S_0=\emptyset$ and $S_k=\{1,2,\ldots,k\}$ for positive integers $k$.  In 2022, Banerjee et al. \cite[Thms. 1.5, 1.7]{banerjee2022parity} considered how the parity bias in \eqref{parityintro} is impacted by the restriction of allowed parts of small size. In particular, they prove that
\begin{align}
    p_e^{S_1}(n)>p_o^{S_1}(n), \text{ for all } n\geq 8, \label{banerjee1}\\
    p_o^{S_2}(n)>p_e^{S_2}(n), \text{ for all } n\geq 9.  \label{banerjee2}
\end{align}
We note that these results align with one's intuitive expectation in this setting, as the smallest allowed part will appear more often overall, and this smallest part is even in the first case, and odd in the second case.  

Let $r_o(n)$, $r_e(n)$, $r_o^S(n)$, and $r_e^S(n)$ denote the fixed perimeter analogues of $p_o(n)$, $p_e(n)$, $p_o^S(n)$, and $p_e^S(n)$, respectively.  

In order to track fixed degrees of bias between even and odd parts, we make the following definitions.  For a partition $\pi$, set $o(\pi)$ and $e(\pi)$ to denote the number of odd and even parts in $\pi$, respectively.  We then define
\begin{equation*}
\omega(\pi) = o(\pi) - e(\pi),
\end{equation*}  
which tracks the amount of bias toward odd parts over even parts.  Let $P(n,m)$ be the set of partitions $\pi$ of $n$ with $\omega(\pi)=m$, and let $p(n,m)$ count the number of partitions in $P(n,m)$.  Let $R(n,m)$ and $r(n,m)$ denote the fixed perimeter analogues of $P(n,m)$ and $p(n,m)$, respectively.

Our first result is a fixed perimeter parity bias generalization of \eqref{parityintro}--\eqref{banerjee2} which uses $r(n,0)$ to specify the difference. 

\begin{theorem}\label{restrictedparts}
For $k\geq 0$, and $n \geq k+3$ or $n=k+1$, 
\begin{align*}
r_o^{S_k}(n) &> r_e^{S_k}(n), \text{ if $k$ even}, \\
r_e^{S_k}(n) &> r_o^{S_k}(n), \text{ if $k$ odd}.
\end{align*}
Specifically, for these $n$ we have $r_o^{S_k}(n) - r_e^{S_k}(n) = \pm r^{S_k}(n, 0)$.
For all other $n$ we have equality.
\end{theorem}

We note that as special cases of Theorem \ref{restrictedparts} we obtain the following fixed perimeter analogues of \eqref{parityintro}--\eqref{banerjee2},
\begin{align*}
    r_o(n) &> r_e(n), \text{ for all } n\geq 3 \text{ and } n=1, \\
    r_e^{S_1}(n) &> r_o^{S_1}(n), \text{ for all } n\geq 4 \text{ and } n=2, \\
    r_o^{S_2}(n) &> r_e^{S_2}(n), \text{ for all } n\geq 5 \text{ and } n=3. 
\end{align*}

Let $S_2^- =\{2\}$ and $S_k^-=\{1,2,\ldots,k-2, k\}$ for positive integers $k\geq 3$.  Banerjee et al. \cite[Thm. 1.6]{banerjee2022parity} also proved that for all $n\geq 1$,
\begin{equation}\label{banerjee3}
p_o^{S_2^-}(n)>p_e^{S_2^-}(n).  
\end{equation}

We also show a fixed perimeter parity bias generalization of \eqref{banerjee3}.

\begin{theorem}\label{oneelementremoved}
For $k\geq 2$ and $n \geq k-1$, 
\begin{align*}
r_o^{S_k^-}(n) &> r_e^{S_k^-}(n), \text{ if $k$ even}, \\
r_e^{S_k^-}(n) &> r_o^{S_k^-}(n), \text{ if $k$ odd}.
\end{align*}
For all other $n$ we have a trivial equality.
\end{theorem}

As a special case of Theorem \ref{oneelementremoved} we obtain the following fixed perimeter analogue of \eqref{banerjee3}.  Namely, for all $n\geq 1$, 
\[
r_o^{S_2^-}(n) > r_e^{S_2^-}(n). 
\]

Considering partitions with a fixed degree of bias between even and odd parts, Kim and Kim \cite[Thm. 1]{kim2023refined} showed that for all $n\geq 14$ and for $n=1,3,6,8,10,12$,
\begin{equation} \label{kimkimdegreeofbias}
p(n,1) > p(n,-1).
\end{equation}

We find that the direction of inequality in the fixed perimeter analogue of \eqref{kimkimdegreeofbias} depends on the parity of $n$.  In the following result we generalize this result to any choice of $m$. 

\begin{theorem}\label{degreeofbias}
Let $n \geq m \geq 1$. Then,
\begin{align*}
r(n,m) &> r(n,-m)  \text{ if } n\equiv m \!\!\! \pmod{2}, \\
r(n,-m) &> r(n,m)  \text{ if } n\not \equiv m \!\!\! \pmod{2}.
\end{align*}
\end{theorem}

As a special case of Theorem \ref{degreeofbias} we obtain the following fixed perimeter version of \eqref{kimkimdegreeofbias}.  Namely, 
\begin{align*}
r(n,1) &> r(n,-1), \text{ for all } n\geq 1 \text{ odd}, \\
r(n,-1) &> r(n,1), \text{ for all } n\geq 2 \text{ even}.
\end{align*}

It is natural to also consider parity bias in the context of partitions into distinct parts.  Let
\begin{align*}
    pd_o(n)\coloneqq p(n \mid \text{distinct parts and more odds than evens}),   \\
    pd_e(n)\coloneqq p(n \mid \text{distinct parts and more evens than odds}), 
\end{align*}
and let $rd_o(n)$ and $rd_e(n)$ denote the fixed perimeter analogues of $pd_o(n)$ and $pd_e(n)$, respectively.

Originally conjectured by Kim, Kim, and Lovejoy in \cite{kim2020parity}, Banerjee et al. \cite[Thm 1.4]{banerjee2022parity} prove the following parity bias result for partitions into distinct parts.  Namely, for all $n \geq 20$,
\begin{equation}\label{banerjee4}
pd_o(n)>pd_e(n).
\end{equation}
While we have not proven a fixed perimeter analogue of \eqref{banerjee4}, we make the following conjecture which has been verified up to $n=150$. 

\begin{conjecture}\label{oddevendistconj}
For all $n \geq 9$,
\[
rd_o(n) > rd_e(n).
\]
\end{conjecture}

\noindent We note that for $n\leq 8$, the inequality in Conjecture \ref{oddevendistconj} holds for odd $n$ but not for even $n$.

Lastly, we consider bias when counting partitions into parts requiring parts of a certain parity to be distinct but allowing parts of the other parity to be unrestricted.  Let
\begin{align*}
    ped(n)\coloneqq p(n \mid \text{even parts distinct}),   \\
    pod(n)\coloneqq p(n \mid \text{odd parts distinct}). 
\end{align*}
These partition counting functions have been widely studied \cite{andrews2009partitions, andrews2025identities, andrews2010arithmetic, ballantine20234, ballantine2025inclusion, ballantine2023ped, ballantine2024generalizations, merca2017new, sellers2025explaining, toh2012ramanujan}. 

In particular, it follows from Ballantine and Welch \cite[Thm. 1.5]{ballantine2023ped} that for all $n\geq 2$,
\begin{equation}\label{eq:pedpod}
ped(n) > pod(n).
\end{equation}

We further refine these functions by excluding parts from a set $S$,
\begin{align*}
    ped^S(n)\coloneqq p(n \mid \text{even parts distinct and no parts from } S),   \\
    pod^S(n)\coloneqq p(n \mid \text{odd parts distinct and no parts from } S). 
\end{align*}
Let $red(n)$, $rod(n)$, $red^S(n)$, and $rod^S(n)$ denote the fixed perimeter analogues of $ped(n)$, $pod(n)$, $ped^S(n)$, and $pod^S(n)$, respectively.  

We obtain results for these functions analogous to Theorems \ref{restrictedparts} and \ref{oneelementremoved}.  

\begin{theorem}\label{restrictedpartsdist}
For $k\geq 0$, and $n \geq k+4$ or $n=k+2$, 
\begin{align*}
red^{S_k}(n) &> rod^{S_k}(n), \text{ if $k$ even}, \\
rod^{S_k}(n) &> red^{S_k}(n), \text{ if $k$ odd}.
\end{align*}
For all other $n$ we have equality.
\end{theorem}

\begin{theorem}\label{oneelementremoveddist}
For $k\geq 2$ and $n \geq k$, 
\begin{align*}
red^{S_k^-}(n) &> rod^{S_k^-}(n), \text{ if $k$ even}, \\
rod^{S_k^-}(n) &> red^{S_k^-}(n), \text{ if $k$ odd}.
\end{align*}
\end{theorem}

Letting $k=0$ in Theorem \ref{restrictedpartsdist} gives the following immediate corollary which is an analogue to \eqref{eq:pedpod}.

\begin{corollary}\label{redrod}
For all $n \geq 4$ and $n=2$, 
\begin{align*}
red(n) > rod(n).
\end{align*}
\end{corollary}

Let $rd(n)$ count the number of partitions with perimeter $n$ into distinct parts. When Straub \cite{straub2016core} proved \eqref{Euleranalogue}, he showed that in fact
\[
rd(n) = F_n,
\]
where $F_n$ is the $n$th Fibonacci number ($F_0=0$, $F_1=1$, and $F_{n+2}=F_{n+1}+F_n$ for all $n\geq 0$).  Letting $rd^S(n)$ count the number of partitions with perimeter $n$ into distinct parts not in $S$, we show in Section \ref{Fib} that Straub's proof extends to show that for any $k\geq 0$ and $n\geq k+1$,
\begin{equation}\label{Straubextn}
rd^{S_k}(n) = F_{n-k}.
\end{equation}

Moreover, we further show in Section \ref{Trib} that the $red$ and $rod$ functions satisfy Tribonacci recurrence relations ($T_{n+3}=T_{n+2}+T_{n+1}+T_n$ for all $n\geq 0$).  

We now outline the remainder of this paper.  In Section \ref{sec:restrictionofallowedparts}, we prove Theorems \ref{restrictedparts} and \ref{oneelementremoved}. In Section \ref{sec:degreeofbias}, we prove Theorem \ref{degreeofbias}.  In Section \ref{sec:redrod}, we prove Theorems \ref{restrictedpartsdist} and \ref{oneelementremoveddist} as well as \eqref{Straubextn} and the Tribonacci recurrence relations for $red(n)$ and $rod(n)$.

\section{Parity Bias for Partitions with Restriction of Allowed Parts}\label{sec:restrictionofallowedparts}

For $k\geq 1$, let $r_k(n)$ count the number of partitions with perimeter $n$ and smallest part equal to $k$, and let $R_k(n)$ denote the set of partitions counted by $r_k(n)$.  Before we prove Theorem \ref{restrictedparts} we introduce a useful bijection.

For $n\geq k+1$, we define the map $$\varphi_k: R_k(n) \rightarrow R^{S_k}(n)$$ by assigning $\varphi_k(\pi)$ to be the partition obtained from $\pi \in R_k(n)$ by deleting one part of size $k$  and increasing all remaining parts by $1$.  Since $n\geq k+1$, there must be a part remaining after deleting a part of size $k$, so this is well-defined.  The perimeter of $\varphi_k(\pi)$ remains $n$, and the parts of $\varphi_k(\pi)$ are all $\geq k+1$.  Thus $\varphi_k(\pi)\in R^{S_k}(n)$.  We construct an inverse by defining $$\psi_k: R^{S_k}(n) \rightarrow R_k(n)$$ to reduce the size of all parts by $1$ and then add a part of size $k$.  Again, since $n\geq k+1$ this is well-defined.  As $\varphi_k$ and $\psi_k$ are inverses, we deduce that for all $k\geq 1$ and $n\geq k+1$,
\begin{equation}\label{bigbijection}
r_k(n) = r^{S_k}(n).
\end{equation}

Observing that for $k\geq 1$, $r_k(n)=r^{S_{k-1}}(n)-r^{S_k}(n)$, we obtain that for all $k\geq 0$,
\begin{equation}\label{bijectioncor}
r^{S_k}(n) = 2r^{S_{k+1}}(n).
\end{equation}
Since $r^{S_{n-1}}(n)=1$, \eqref{bijectioncor} gives a refinement of the fact that $r(n)=2^{n-1}$, which was shown by Fu and Tang in \cite[Cor. 2.4]{fu2018partitions}.

As $R^{S_k}(n) = R^{S_{k+1}}(n) \sqcup R_{k+1}(n)$, putting the maps $\psi_{k+1}$ and $\varphi_{k+1}$ together, yields the following useful bijection for $k\geq 0$ and $n\geq k+2$,
\[
\Theta: R^{S_k}(n) \rightarrow R^{S_k}(n), 
\]
defined by 
\begin{equation}\label{eq:Thetadef}
\Theta(\pi) = 
\begin{cases}
\psi_{k+1}(\pi) & \text{ if } \pi \in R^{S_{k+1}}(n), \\
\varphi_{k+1}(\pi) & \text{ if } \pi \in R_{k+1}(n).
\end{cases}
\end{equation}
We observe that $\Theta^{-1}=\Theta$, so $\Theta$ is an involution.

In order to consider how $\Theta$ behaves on parity bias partitions, we make the following definitions.  Recall $\omega(\pi)$ tracks the amount of bias toward odd parts over even parts in a partition $\pi$.  For a set $S$ we define
\[
R^S(n,m) = \{ \pi \in R(n) \mid \omega(\pi)=m \text{ and no parts come from } S \},
\]
and $r^S(n,m)$ to be the number of partitions in $R^S(n,m)$.  Then, 
\begin{align}
R_o^S(n) &= \bigsqcup_{m\geq 1} R^S(n,m), \label{eq:oddrefinement} \\
R_e^S(n) &= \bigsqcup_{m\geq 1} R^S(n,-m). \label{eq:evenrefinement}
\end{align}

For $k\geq 1$, we can restrict $\psi_k$ to $R^{S_k}(n,m)$, which yields a bijection
\begin{equation}\label{eq:psi_kbij}
\psi_k: R^{S_k}(n,m) \rightarrow
\begin{cases}
R_k(n, -m-1) & \text{ if $k$ even}, \\
R_k(n, -m+1) & \text{ if $k$ odd}.
\end{cases} 
\end{equation}
This is because if $\omega(\pi) - o(\pi) - e(\pi) =m$, reducing the parts of $\pi$ by $1$ changes the parity of each part and thus reverses the sign on $m$.  Then adding a part of size $k$ alters the parity bias by $\pm 1$ depending on whether $k$ is odd or even.  The restriction of $\varphi_k$ to $R_k(n, -m-1)$ if $k$ is even, or to $R_k(n, -m+1)$ if $k$ is odd, is an inverse of $\psi_k$.  Thus, 
\begin{equation}\label{eq:varphi_kbij}
\varphi_k: R_k(n,m) \rightarrow
\begin{cases}
R^{S_k}(n, -m-1) & \text{ if $k$ even}, \\
R^{S_k}(n, -m+1) & \text{ if $k$ odd},
\end{cases} 
\end{equation}
and we have that for each $k\geq 1$, $n\geq k+1$, and $m\in \mathbb{Z}$, 
\begin{align*}
r^{S_k}(n,m) &= r_k(n, -m-1),  \text{ if $k$ even}, \\
r^{S_k}(n,m) &= r_k(n, -m+1),  \text{ if $k$ odd}.
\end{align*} 
We note that for all but finitely many $m$, these values will be $0$ as the corresponding sets will be empty.

Together, \eqref{eq:psi_kbij} and \eqref{eq:varphi_kbij} give that $\Theta$ restricts to a bijection 
\begin{equation}\label{eq:Thetarestriction}
\Theta : R^{S_k}(n,m) \rightarrow 
\begin{cases}
R^{S_k}(n, -m+1) & \text{ if $k$ even}, \\
R^{S_k}(n, -m-1) & \text{ if $k$ odd}.
\end{cases} 
\end{equation}
for any $k\geq 0$, $n\geq k+2$, and $m\in \mathbb{Z}$.  Here the parity conditions have reversed since $R^{S_k}(n,m) = R^{S_{k+1}}(n,m) \sqcup R_{k+1}(n,m)$.

We now prove Theorem \ref{restrictedparts} using the function $\Theta$.

\begin{proof}[Proof of Theorem \ref{restrictedparts}]
Let $k\geq 0$.  We first note that the partition $k+1$ is the only partition with perimeter $k+1$ that contains no parts from $S_k$.  Thus,
\[
r_o^{S_k}(k+1) = 
\begin{cases}
1 & \text{if $k$ even}, \\
0 & \text{if $k$ odd},
\end{cases}
\]
and so the result holds when $n=k+1$.  

We also note that there are exactly two partitions with perimeter $k+2$ that have no parts in $S_k$, these are $k+2$, and $(k+1, k+1)$.  Since the parities of these parts are opposite, it follows that $r_o^{S_k}(k+2) = r_e^{S_k}(k+2)$ for $k\geq 0$ of any parity.  Moreover, if $n\leq k$ we trivially have $r_o^{S_k}(k+2) = r_e^{S_k}(k+2) =0$. 

We now suppose $n\geq k+3$.  We observe that the partition $\pi = (n-1, n-2)$ of perimeter $n$ has one even and one odd part with each part $\geq k+1$.  Thus for all $k\geq 0$, and $n\geq k+3$, 
\begin{equation}\label{eq:nonempty}
R^{S_k}(n,0) \neq \emptyset.
\end{equation}
We will use \eqref{eq:nonempty} to prove the result.

Consider first the case when $k\geq 0$ is even.  Let $\Theta_o$ be the restriction of $\Theta$ to $R_o^{S_k}(n)$.

From \eqref{eq:oddrefinement}, \eqref{eq:psi_kbij}, and \eqref{eq:varphi_kbij}, since $k$ is even (and thus $k+1$ is odd), we have that $\Theta_o$ is a bijection onto
\begin{align*}
\Theta_o(R_o^{S_k}(n)) 
&= \bigsqcup_{m\geq 1} \left( R^{S_{k+1}}(n, -m+1)  \bigsqcup R_{k+1}(n, -m+1) \right) \\
& = \bigsqcup_{m\geq 1}  R^{S_k}(n, -m+1) \\
& = R^{S_k}(n, 0) \sqcup R_e^{S_k}(n).
\end{align*}
It thus follows that
\[
r_o^{S_k}(n) = r_e^{S_k}(n) + |R^{S_k}(n, 0)|,
\]
and from \eqref{eq:nonempty} we have that $r_o^{S_k}(n) > r_e^{S_k}(n)$ in this case.

We now consider the case when $k\geq 1$ is odd.  Similar to the even case, we define $\Theta_e$ to be the restriction of $\Theta$ to $R_e^{S_k}(n)$.

Since $k$ is now odd (and $k+1$ is even), we have from \eqref{eq:evenrefinement}, \eqref{eq:psi_kbij}, and \eqref{eq:varphi_kbij}, that $\Theta_e$ is a bijection onto
\begin{align*}
\Theta_e(R_e^{S_k}(n)) &= \bigsqcup_{m\geq 1} \left( R^{S_{k+1}}(n, m-1)  \bigsqcup R_{k+1}(n, m-1) \right) \\
& = R^{S_k}(n, 0) \sqcup R_o^{S_k}(n).
\end{align*}
Thus,
\[
r_e^{S_k}(n) = r_o^{S_k}(n) + |R^{S_k}(n, 0)|,
\]
and from \eqref{eq:nonempty} we also have that $r_e^{S_k}(n) > r_o^{S_k}(n)$ in this case.
\end{proof}

Before we prove Theorem \ref{oneelementremoved} we establish an involution on $R^{S_k^-}(n)$.  For a partition $\pi \in R^{S_k^-}(n)$, define $\lambda_{k,\pi} \geq 0$ to be the number of parts of size $k-1$ in $\pi$, and let $\pi'$ be the partition obtained from $\pi$ by removing the $\lambda_{k,\pi}$ parts of size $k-1$.  Then $\pi' \in R^{S_k}(n-\lambda_{k,\pi})$, and we observe that $\Theta(\pi')$ is well-defined if and only if $n \geq k+2+\lambda_{k,\pi}$.  

Let $k\geq 2$ and $n\geq k+2$.  If $\pi\in R^{S_k^-}(n)$ has largest part $\geq k+2$ and $\lambda_\pi$ parts, then $\lambda_\pi \geq \lambda_{k,\pi} + 1$, so the perimeter $n \geq k+2+\lambda_{k,\pi}$ and $\Theta(\pi')$ is well-defined.  Likewise, if $\pi\in R^{S_k^-}(n)$ has largest part $k+1$ and $\lambda_\pi \geq \lambda_{k,\pi} + 2$ parts, then $n=k + 2 + \lambda_{k,\pi}$ and $\Theta(\pi')$ is well-defined.  Thus the only partitions $\pi \in R^{S_k^-}(n)$ such that $\Theta(\pi')$ is not well-defined are 
\begin{align*}
\pi_0 &= (k-1)^{n-k+2}, \\
\pi_1 & = (k+1, (k-1)^{n-k-1}).
\end{align*}
The partition $\pi_1$ is the unique partition in $R^{S_k^-}(n)$ that has largest part $k+1$ and $\lambda_\pi = \lambda_{k,\pi} + 1$ parts ($n=k +1 + \lambda_{k,\pi}$), and $\pi_0$ is the unique partition in $R^{S_k^-}(n)$ that has largest part $k-1$ ($n = k -2 + \lambda_{k,\pi}$).  

Thus for $k\geq 2$ and $n\geq k+2$, we define the involution
\begin{equation}\label{eq:Theta-def}
\Theta^-: R^{S_k^-}(n) \rightarrow R^{S_k^-}(n)
\end{equation}
by setting $\Theta^-(\pi_0)=\pi_0$, $\Theta^-(\pi_1)=\pi_1$, and for all other $\pi \in R^{S_k^-}(n)$, we define $\Theta^-(\pi)$ to be the partition obtained by fixing each part of size $k-1$ and applying $\Theta$ to $\pi'$ to obtain a partition in $R^{S_k^-}(n)$.  

We now prove Theorem \ref{oneelementremoved} using the function $\Theta^-$.

\begin{proof}[Proof of Theorem \ref{oneelementremoved}]
Let $k\geq 2$ and $n\geq k-1$.  We note that if $n\leq k-2$ then $r_e^{S_k^-}(n)=r_o^{S_k^-}(n)=0$.  Moreover, the only partition in $R^{S_k^-}(k-1)$ is $k-1$, the only partition in $R^{S_k^-}(k)$ is $(k-1)^2$, and the only partitions in $R^{S_k^-}(k)$ are $k+1$ and $(k-1)^3$.  In each of these three cases, the partitions are all in $R_o^{S_k^-}(n)$ if and only if $k$ is even and in $R_e^{S_k^-}(n)$ if and only if $k$ is odd.  Thus the result holds for $n\leq k+1$.  We now suppose $n\geq k+2$.

First assume $k\geq 2$ is even.  Let $\Theta_e^-$ be the restriction of $\Theta^-$ to $R_e^{S_k^-}(n)$.  Then $\pi_0, \pi_1 \not\in R_e^{S_k^-}(n)$.  So for all $\pi\in R_e^{S_k^-}(n)$, $\Theta_e^-(\pi)$ is obtained by fixing each $k-1$ and applying $\Theta$ to $\pi'$.  Since $k$ is even and $\omega(\pi)\leq -1$, we have that $\omega(\pi') = \omega(\pi) - \lambda_{k,\pi} \leq -1$ as well.  So 
\[
\pi' \in R_e^{S_k}(n-\lambda_{k,\pi}) = \bigsqcup_{m\geq 1} R^{S_k}(n-\lambda_{k,\pi},-m).
\] 
By \eqref{eq:Thetarestriction}, we obtain
\[
\Theta(\pi') \in \bigsqcup_{m\geq 1} R^{S_k}(n-\lambda_{k,\pi}, m+1) \subseteq R_o^{S_k}(n-\lambda_{k,\pi}),
\] 
and thus $\Theta_e^-(\pi)\in R_o^{S_k^-}(n)$.  So for all $k\geq 2$ even and $n\geq k+1$, 
\[
\Theta_e^-: R_e^{S_k^-}(n) \hookrightarrow R_o^{S_k^-}(n).
\]
Moreover, $\Theta_e^-$ is not surjective since $\pi_0, \pi_1$ are in $R_o^{S_k^-}(n)$ but not in $R_e^{S_k^-}(n)$ when $k$ is even.  The result follows in this case.

We now assume $k\geq 3$ is odd, and let $\Theta_o^-$ be the restriction of $\Theta^-$ to $R_o^{S_k^-}(n)$.  Here, $\pi_0, \pi_1 \not\in R_o^{S_k^-}(n)$, so for all $\pi\in R_o^{S_k^-}(n)$, $\Theta_o^-(\pi)$ is obtained by fixing each $k-1$ and applying $\Theta$ to $\pi'$.  Since $k$ is odd and $\omega(\pi) \geq 1$, we have that $\omega(\pi') = \omega(\pi) + \lambda_{k,\pi} \geq 1$ as well.  So $\pi' \in R_o^{S_k}(n-\lambda_{k,\pi})$, and by \eqref{eq:Thetarestriction}, we obtain
\[
\Theta(\pi') \in \bigsqcup_{m\geq 1} R^{S_k}(n-\lambda_{k,\pi}, -m-1) \subseteq R_e^{S_k}(n-\lambda_{k,\pi}).
\] 
Thus $\Theta_o^-(\pi)\in R_e^{S_k^-}(n)$, and for all $k\geq 2$ even and $n\geq k+1$, 
\[
\Theta_o^-: R_o^{S_k^-}(n) \hookrightarrow R_e^{S_k^-}(n).
\]
Moreover, $\Theta_o^-$ is not surjective since now $\pi_0, \pi_1$ are in $R_e^{S_k^-}(n)$ but not in $R_o^{S_k^-}(n)$ since $k$ is odd.  The result follows.
\end{proof}

\section{Parity Bias for Partitions with Fixed Degree of Bias}\label{sec:degreeofbias}

Recall from the introduction that if $\pi \in R(n)$, then $\alpha_\pi$ denotes the largest part, $\lambda_\pi$ the number of parts, and $n=\alpha_\pi+\lambda_\pi -1$.  We prove Theorem \ref{degreeofbias} by employing a stars and bars counting technique to obtain direct formulas for $r(n,m)$ that we can then compare.  Throughout we adopt the convention that binomial coefficients $\binom{n}{k}:=0$ when $k<0$ or $n<k$. 

\begin{proof}[Proof of Theorem \ref{degreeofbias}]

We first note that as $1^n$ is the only partition with perimeter $n$ into more than $n-1$ parts, $r(n,n)=1$, while $r(n,-n)=0$.  Thus the result holds when $n=m$.  

We now consider when $n>m\geq 1$.  Then for $\pi\in R(n, \pm m)$, we must have that $\alpha_\pi \geq 2$.  We further observe that since $\lambda_\pi = \omega(\pi) +2 e(\pi)$, it follows that  $\lambda_\pi \equiv \omega(\pi) \pmod{2}$.  Thus we have $\alpha_\pi \equiv n-\omega(\pi)+1 \pmod{2}$, so 
\[
\alpha_\pi \equiv 
\begin{cases}
1 \!\!\! \pmod{2} & \text{ if } n\equiv \omega(\pi) \!\!\! \pmod{2}, \\
0 \!\!\! \pmod{2} & \text{ if } n \not\equiv \omega(\pi) \!\!\! \pmod{2}.
\end{cases}
\]
Furthermore, since $|\omega(\pi)| \leq \lambda_\pi = n - \alpha_\pi +1$, it follows that 
\begin{equation}\label{eq:nbound}
n \geq |\omega(\pi)| + \alpha_\pi -1.
\end{equation}

Using a stars and bars counting technique, the number of ways to choose $k$ elements from a set of size $N$ allowing repetition of elements is
\begin{equation}\label{eq:starsbars}
\binom{N+k-1}{k}.
\end{equation}

We use \eqref{eq:starsbars} to count the number of ways to construct an element of $R(n,m)$ with fixed largest part $\alpha$ by choosing the additional $\lambda -1 = n - \alpha$ parts separately, based on whether they are odd or even.  Let $A_1$ denote the number of additional odd parts, and $A_2$ the number of additional even parts.  Then since $A_1+A_2= n - \alpha$ and 
\[
A_1-A_2 = 
\begin{cases} 
m-1 & \text{ if $\alpha$ odd}, \\ 
m+1 & \text{ if $\alpha$ even}, 
\end{cases}
\]
it follows that 
\[
A_1 = \begin{cases} \frac{n + m -\alpha -1}{2} & \text{ if $\alpha$ odd}, \\ \frac{n + m -\alpha +1}{2} & \text{ if $\alpha$ even}, \end{cases}
\]
and 
\[
A_2 = \begin{cases} \frac{n - m -\alpha +1}{2} & \text{ if $\alpha$ odd}, \\ \frac{n -m -\alpha -1}{2} & \text{ if $\alpha$ even}. \end{cases}
\]
If $\alpha$ is even, then $m<\lambda$ and $n\not\equiv m \pmod{2}$.  Thus with \eqref{eq:nbound},
\begin{equation}\label{eq:nbdds}
n\geq
\begin{cases}
m+\alpha -1 & \text{ if $\alpha$ odd}, \\
m+\alpha +1 & \text{ if $\alpha$ even},
\end{cases}
\end{equation}
and it follows that $A_1, A_2 \geq 0$.

Since the number of odd and even positive integers up to $\alpha$ is $\lfloor \frac{\alpha +1}{2} \rfloor$ and $\lfloor \frac{\alpha}{2} \rfloor$, respectively, \eqref{eq:starsbars} gives that the total number of partitions in $R(n,m)$ with largest part $\alpha$ is
\begin{align}
\binom{\frac{n+m}{2} -1}{\frac{n+ m}{2} - \frac{\alpha +1}{2}} \binom{\frac{n-m}{2} -1}{\frac{n- m}{2} - \frac{\alpha -1}{2}} &  \text{ if $\alpha$ odd}, \label{eq:choose1} \\ 
\binom{\frac{n+m+1}{2}-1}{\frac{n + m +1}{2} - \frac{\alpha}{2}} \binom{\frac{n-m-1}{2} -1}{\frac{n - m -1}{2} - \frac{\alpha}{2}} & \text{ if $\alpha$ even}. \label{eq:choose2} 
\end{align}
Thus we obtain formulas for $r(n,m)$ depending on whether (or not) $n$ and $m$ have the same parity by summing over all possible $\alpha = 2k+1 \leq n-m+1$ (or $\alpha=2k \leq n-m-1$) according to \eqref{eq:nbdds}.  Since we have established that $\alpha \geq 2$, this gives that
\begin{equation*}\label{eq:r(n,m)}
r(n,m) = 
\begin{cases}
\displaystyle \sum_{k=1}^{\frac{n-m}{2}}  \binom{\frac{n+m}{2} -1}{\frac{n+ m}{2} -k -1} \binom{\frac{n-m}{2} -1}{\frac{n- m}{2} -k} & \text{ if } n\equiv m \!\!\! \pmod{2}, \\
\displaystyle \sum_{k=1}^{\frac{n-m-1}{2}} \binom{\frac{n+m+1}{2}-1}{\frac{n + m +1}{2} - k} \binom{\frac{n-m-1}{2} -1}{\frac{n - m -1}{2} - k} & \text{ if } n\not\equiv m \!\!\! \pmod{2}.
\end{cases}
\end{equation*}
Similarly, the total number of partitions in $R(n,-m)$ with largest part $\alpha$ is obtained from \eqref{eq:choose1} and \eqref{eq:choose2} by replacing $m$ with $-m$,\begin{align*}
\binom{\frac{n-m}{2} -1}{\frac{n- m}{2} - \frac{\alpha +1}{2}} \binom{\frac{n+m}{2} -1}{\frac{n+ m}{2} - \frac{\alpha -1}{2}} &  \text{ if $\alpha$ odd}, \\ 
\binom{\frac{n-m+1}{2}-1}{\frac{n - m +1}{2} - \frac{\alpha}{2}} \binom{\frac{n+m-1}{2} -1}{\frac{n + m -1}{2} - \frac{\alpha}{2}} & \text{ if $\alpha$ even}. 
\end{align*}
In this case, if $\alpha$ is odd, then $m<\lambda$ and $n\equiv m \pmod{2}$.  Thus with \eqref{eq:nbound},
\begin{equation}\label{eq:nbdds}
n\geq
\begin{cases}
m+\alpha +1 & \text{ if $\alpha$ odd}, \\
m+\alpha -1 & \text{ if $\alpha$ even}.
\end{cases}
\end{equation}
We thus have that
\begin{equation*}
r(n,-m) = 
\begin{cases}
\displaystyle \sum_{k=1}^{\frac{n-m-2}{2}}  \binom{\frac{n-m}{2} -1}{\frac{n-m}{2} -k -1} \binom{\frac{n+m}{2} -1}{\frac{n+m}{2} -k} & \text{ if } n\equiv m \!\!\! \pmod{2}, \\
\displaystyle \sum_{k=1}^{\frac{n-m+1}{2}} \binom{\frac{n-m+1}{2}-1}{\frac{n - m +1}{2} - k} \binom{\frac{n+m-1}{2} -1}{\frac{n + m -1}{2} - k} & \text{ if } n\not\equiv m \!\!\! \pmod{2}.
\end{cases}
\end{equation*}

Using the fact that $\binom{a}{b+1} = \frac{a-b}{b+1}\binom{a}{b}$, it follows that when $n\equiv m \pmod{2}$ and $k<\frac{n-m}{2}$, the $k$th summand of $r(n,m)$ is equal to $(\frac{n+m}{2} -k)/(\frac{n-m}{2} - k) >1$ times the $k$th summand of $r(n,m)$.  Thus the result follows when $n \equiv m \pmod{2}$.  

Similarly, using the fact that $\binom{a+1}{b+1} = \frac{a+1}{b+1}\binom{a}{b}$, it follows that when $n \not\equiv m \pmod{2}$, the $k$th summand of $r(n,m)$ is equal to 
\begin{equation}\label{eq:factor}
\left( \frac{\frac{n-m+1}{2} -k}{\frac{n-m+1}{2} - 1} \right) \left( \frac{\frac{n+m+1}{2} -1}{\frac{n+m+1}{2} - k} \right)
\end{equation}
times the $k$th summand of $r(n,m)$.  When  $x>y\geq 1$ it follows that $\frac{x+m}{y+m} < \frac{x}{y}$.  Thus, for $k=1$ \eqref{eq:factor} is $1$, but when $k>1$, 
\[
\left( \frac{\frac{n+m+1}{2} -1}{\frac{n+m+1}{2} - k} \right) < \left( \frac{\frac{n-m+1}{2} -1}{\frac{n-m+1}{2} - k} \right),  
\]
so \eqref{eq:factor} is $<1$.  Since $r(n,-m)$ has one more summand than $r(n,m)$, the result follows when $n \not\equiv m \pmod{2}$ which completes the proof.
\end{proof}

\section{Analogues of PED and POD Partitions}\label{sec:redrod}
  
We can use the same approach of restricting the functions $\Theta$ and $\Theta^-$ as in the proofs of Theorems \ref{restrictedparts} and \ref{oneelementremoved} in Section \ref{sec:restrictionofallowedparts} to prove Theorems \ref{restrictedpartsdist} and \ref{oneelementremoveddist}.  Let $RED^{S_k}(n)$ and $ROD^{S_k}(n)$ denote the sets of partitions counted by $red^{S_k}(n)$ and $rod^{S_k}(n)$, respectively.  
   
\begin{proof}[Proof of Theorem \ref{restrictedpartsdist}]
We first observe that when $n=k+2$, the only partitions in $R^{S_k}(k+2)$ are $k+2$ and $(k+1, k+1)$.  Since $k+2$ has distinct parts it is in both $RED^{S_k}(n)$ and $ROD^{S_k}(n)$.  However, $(k+1, k+1)$ is in $RED^{S_k}(n)$ if and only if $k$ is even, and in $ROD^{S_k}(n)$ if and only if $k$ is odd.  Thus we obtain the desired inequality when $n=k+2$.  

We also observe that when $n=k+3$, the partitions in $R^{S_k}(k+3)$ are $k+3$, $(k+2, k+2)$, $(k+2, k+1)$, and $(k+1, k+1, k+1)$.  Since $k+2$ and $k+1$ are of opposite parity, we have $red^{S_k}(k+3)=rod^{S_k}(k+3)=3$ in this case.  When $n=k+1$ the partition $k+1$ gives that $red^{S_k}(n)=rod^{S_k}(n)=1$.  Moreover, if $n\leq k$, we trivially have $red^{S_k}(n)=rod^{S_k}(n)=0$.

We now assume $n\geq k+4$.  Since $\Theta$ changes the parity of parts (and adds or removes one part of size $k+1$), it follows that
\begin{align*}
\Theta(ROD^{S_k}(n)) \subseteq RED^{S_k}(n) & \text{ if $k$ even}, \\
\Theta(RED^{S_k}(n)) \subseteq ROD^{S_k}(n) & \text{ if $k$ odd}.
\end{align*}
Thus since $\Theta$ is injective, it follows that $red^{S_k}(n) \geq rod^{S_k}(n)$, when $k$ even, and $rod^{S_k}(n) \geq red^{S_k}(n)$, when $k$ odd.

To see that these inequalities are strict, consider the partition $\pi=(n-1,k+2)$.  Then as $n\geq k+4$, we have that $\pi \in RED^{S_k}(n) \cap ROD^{S_k}(n)$.  However, $\pi \in R^{S_{k+1}}(n)$, so $\Theta^{-1}(\pi) = \Theta(\pi)=\psi_{k+1}(\pi)=(n-2, k+1, k+1)$.  Thus when $k$ is even $\Theta^{-1}(\pi) \not\in ROD^{S_k}(n)$, and when $k$ is odd, $\Theta^{-1}(\pi) \not\in RED^{S_k}(n)$.  Thus in both cases $\Theta$ is not surjective.  

\end{proof}

\begin{proof}[Proof of Theorem \ref{oneelementremoveddist}]

We first observe that when $n=k$, the only partition in $R^{S_k^-}(k)$ is $(k-1)^2$, which is in $RED^{S_k^-}(n)$ if and only if $k$ is even, and in $ROD^{S_k^-}(n)$ if and only if $k$ is odd.  Thus the inequality holds for $n=k$.  When $n=k+1$, the only partitions in $R^{S_k^-}(k+1)$ are $k+1$ and $(k-1)^3$.  Since $k+1$ has distinct parts it is in both $RED^{S_k^-}(n)$ and $ROD^{S_k^-}(n)$.  However, $(k-1)^3$ is in $RED^{S_k^-}(n)$ if and only if $k$ is even, and in $ROD^{S_k^-}(n)$ if and only if $k$ is odd.  Thus we obtain the desired inequality when $n=k+1$.  Similarly, when $n=k+2$,  
\[
R^{S_k^-}(k+2) = \{ k+2, (k+1)^2, (k+1, k-1), (k-1)^4 \}.
\] 
Two of these have distinct parts, and the other two are in $RED^{S_k^-}(n)$ if and only if $k$ is even, and in $ROD^{S_k^-}(n)$ if and only if $k$ is odd.  Thus the desired inequality follows when $n=k+2$. 

We also observe that when $n=k-1$, the only partition in $R^{S_k^-}(k-1)$ is $k-1$, which has distinct parts, so $red^{S_k^-}(k-1)=rod^{S_k^-}(k-1)=1$ in this case.  If $n\leq k-2$, we trivially have $red^{S_k^-}(n)=rod^{S_k^-}(n)=0$.

We now assume $n\geq k+3$.  In this case $\pi_0$ and $\pi_1$ both contain repeated parts of size $k-1$.  So if $k$ is even, $\pi_0$ and $\pi_1$ belong to $RED^{S_k^-}(n)$ but not $ROD^{S_k^-}(n)$, so for any $\pi \in ROD^{S_k^-}(n)$, we have that $\Theta^-(\pi)$ is obtained by fixing the parts of size $k-1$ (which is odd, so there can be at most one such part) and applying $\Theta$ to $\pi'$.  Since the odd parts of $\pi$ are distinct and $\Theta$ changes the parity of the parts $\geq k+1$ (and adds or removes one odd part of size $k+1$), this results in a partition with distinct even parts. Similarly, if $k$ is odd, then $\pi_0$ and $\pi_1$ belong to $ROD^{S_k^-}(n)$ but not $RED^{S_k^-}(n)$, and for any $\pi \in RED^{S_k^-}(n)$, it follows that $\Theta^-(\pi)$ has distinct odd parts.  Thus,
\begin{align*}
\Theta^-(ROD^{S_k^-}(n)) \subsetneq RED^{S_k^-}(n) & \text{ if $k$ even}, \\
\Theta^-(RED^{S_k^-}(n)) \subsetneq ROD^{S_k^-}(n) & \text{ if $k$ odd}.
\end{align*}
Since $\Theta^-$ is injective, the desired inequalities follow.
\end{proof}

\subsection{Fibonacci Recurrences}  \label{Fib}

Recall that $rd(n)$ counts the number of partitions with perimeter $n$ into distinct parts, and $F_n$ is the $n$th Fibonacci number.  Let $RD(n)$ denote the set of all partitions into distinct parts with perimeter $n$.  Straub \cite{straub2016core} showed that $rd(n) = F_n$ by the following method.  For $n\geq 3$ he splits 
\begin{equation}\label{RDsplit}
RD(n) = RD_1(n) \sqcup RD_2(n),
\end{equation} 
where $RD_1(n)$ consists of those partitions in $RD(n)$ whose largest part is exactly $1$ more than the next largest (there must be at least two parts), and $RD_2(n)$ consists of those whose largest part is at least $2$ greater than the next largest (or is the unique partition with perimeter $n$ into one part).  

He establishes a bijection $\varphi$ between $RD_1(n)$ and the partitions counted by $rd(n-2)$ by removing the largest part of $\pi \in RD_1(n)$ (or for $\varphi^{-1}$, adding a new largest part of size one greater), and a bijection $\psi$ between $RD_2(n)$ and the partitions counted by $rd(n-1)$ by reducing the size of the largest part by $1$ (or for $\psi^{-1}$, increasing the size of the largest part by $1$).  With \eqref{RDsplit}, this gives the Fibonacci recurrence relation for $n\geq 3$,
\[
rd(n) = rd(n-1) + rd(n-2).
\]
Then a simple induction argument establishes that $rd(n)=F_n$ for all $n\geq 1$.

For fixed $k\geq 0$, recall $rd^S(n)$ counts the number of partitions with perimeter $n$ into distinct parts not in $S$.  Let $RD^S(n)$ denote the set of partitions counted by $rd^S(n)$.  We show that Straub's proof above extends to prove \eqref{Straubextn} as well.  For $k\geq 0$ and $n\geq k+3$, write
\begin{equation}\label{RDasplit}
RD^{S_k}(n) = RD_1^{S_k}(n) \sqcup RD_2^{S_k}(n),
\end{equation} 
where $RD_1^{S_k}(n)$ and $RD_2^{S_k}(n)$ denote the subsets of $RD_1(n)$ and $RD_2(n)$ whose partitions only contain parts $\geq k+1$.  Then the function $\varphi_k$ obtained by restricting $\varphi$ to $RD_1^{S_k}(n)$ is a clearly bijection onto the set of partitions counted by $rd^{S_k}(n-2)$ as the smallest part is never affected.  The function $\psi_k$ obtained by restricting $\psi$ to $RD_2^{S_k}(n)$ is also a bijection onto the set of partitions counted by $rd^{S_k}(n-1)$.  This is because the only time the smallest part is reduced is in the case of the unique partition with perimeter $n$ into one part, and since $n\geq k+3$ reducing the size by $1$ does not create any issues.  These bijections yield the Fibonacci recurrence relation for $n\geq k+3$,
\[
rd^{S_k}(n) = rd^{S_k}(n-1) + rd^{S_k}(n-2).
\]
A simple induction argument then establishes \eqref{Straubextn}, noting that we induct on $m=n-k-1$ and use the base cases $m=0$ ($n=k+1$) and $m=1$ ($n=k+2$).  

We further observe that if $rd_k(n)$ counts the number of partitions with perimeter $n$ into distinct parts with smallest part equal to $k$, then for $k\geq 1$,
\[
rd_k(n)= rd^{S_{k-1}}(n) - rd^{S_k}(n) = F_{n-k+1} - F_{n-k}.
\]  
We thus obtain the following refinement of \eqref{Straubextn}.  For $k\geq 1$ and $n\geq k+2$,
\begin{equation}\label{Straubextnrefinement}
rd_k(n) = F_{n-k-1}.
\end{equation}

\subsection{Tribonacci Recurrences}  \label{Trib}   

We now use a similar approach to that in Section \ref{Fib} to prove the following Tribonacci recurrence relations for $red(n)$ and $rod(n)$.  For $n\geq 4$,
\begin{align}
red(n) &= red(n-1) + red(n-2) + red(n-3), \label{redrecurrence} \\
rod(n) &= rod(n-1) + rod(n-2) + rod(n-3). \label{rodrecurrence}
\end{align} 

Equations \eqref{redrecurrence} and \eqref{rodrecurrence} provide another method to deduce Corollary \ref{redrod}, as we will observe below.  We note that \eqref{redrecurrence} and \eqref{rodrecurrence} can also be obtained using generating functions, which is described in \cite[Thm. 5.9]{reu24proceedings}.  Moreover, explicit counting functions for $red(n)$ and $rod(n)$ are given in \cite[Thm. 5.7]{reu24proceedings} which provide yet another method for deducing Corollary \ref{redrod}.

\begin{proof}[Proof of \eqref{redrecurrence} and \eqref{rodrecurrence}]

Let $RED(n)$, $ROD(n)$ denote the sets of partitions counted by $red(n)$, $rod(n)$, respectively.   Write
\begin{align}
RED(n)= RED_1(n) \sqcup RED_2(n) \sqcup RED_3(n), \label{REDsplit} \\
ROD(n)= ROD_1(n) \sqcup ROD_2(n) \sqcup ROD_3(n), \label{RODsplit}
\end{align}
where 
\begin{align*}
RED_1(n) &= \{\pi \in RED(n) \mid 1 \text{ is a part} \}, \\
RED_2(n) &= \{\pi \in RED(n) \mid 1 \text{ is not a part but } 2 \text{ is}\}, \\
RED_3(n) &= \{\pi \in RED(n) \mid \text{neither } 1 \text{ nor } 2 \text{ is a part} \},
\end{align*}
and 
\begin{align*}
ROD_1(n) &= \{\pi \in ROD(n) \mid 2 \text{ is a part} \}, \\
ROD_2(n) &= \{\pi \in ROD(n) \mid 2 \text{ is not a part but } 1 \text{ is} \}, \\
ROD_3(n) &= \{\pi \in ROD(n) \mid \text{neither } 1 \text{ nor } 2 \text{ is a part} \}.
\end{align*}

Let $n\geq 4$.  We can define a bijection $\varphi_e$ from $RED_1(n)$ to the partitions counted by $red(n-1)$ by removing a part of size $1$ (or for $\varphi_e^{-1}$, by appending a part of size $1$).  This reduces (or increases) the perimeter by exactly $1$ while maintaining the distinctness of even parts, and as the mapping is invertible it is a bijection.  Similarly, we define the bijection $\varphi_o$ from $ROD_1(n)$ to the partitions counted by $rod(n-1)$ by removing a part of size $2$ (or for $\varphi_o^{-1}$, by appending a part of size $2$).  

We next define a bijection $\psi_e$ from $RED_2(n)$ to the set of partitions counted by $red(n-3)$ by removing the unique part of size $2$ and then reducing all other parts by $2$.  Reducing parts by $2$ maintains their parity, and this mapping reduces the perimeter by $3$.  Moreover, $\psi_e$ is easily invertible by increasing all parts by $2$ and then appending a part of size $2$.  Similarly, we define the bijection $\psi_o$ from $ROD_2(n)$ to the set of partitions counted by $rod(n-3)$ by removing the unique part of size $1$ and then reducing all other parts by $2$. 

Finally, we define a bijection $\theta_e$ from $RED_3(n)$ to the set of partitions counted by $red(n-2)$ by simply reducing all parts by $2$.  This maintains the parity of each part and reduces the perimeter by $2$.  For the inverse, simply increase all parts by $2$.  Similarly, the same mapping defines a bijection $\theta_o$ from $ROD_3(n)$ to the set of partitions counted by $rod(n-2)$. 

Together with \eqref{REDsplit} and \eqref{RODsplit}, this proves \eqref{redrecurrence} and \eqref{rodrecurrence}.  

\end{proof}

Considering the first three values of $red(n)$ and $rod(n)$ gives us an additional proof of Corollary \ref{redrod}.

\begin{proof}[Tribonacci Proof of Corollary \ref{redrod}]
Observe that $red(1)=rod(1)=1$, as $1$ is the only partition with perimeter $1$ and it has distinct parts.  However, $red(2)=2$ and $rod(2)=1$, since both partitions with perimeter $2$ have distinct even parts but only one has distinct odd parts.  Additionally, $red(3)=rod(3)=3$ as the partitions with perimeter $3$ are $3$, $2+1$, $2+2$, and $1+1+1$, and for each parity there is only one with a repeated part of that parity.  

Thus with these base cases we obtain from \eqref{redrecurrence} and \eqref{rodrecurrence} that $red(n) > rod(n)$ for all $n \geq 4$ and $n=2$, and $red(n)=rod(n)$ for $n=1,3$. 
\end{proof}

\end{document}